\documentclass[runningheads]{llncs}
\usepackage{graphicx}
\usepackage[disable]{todonotes}
\usepackage{amssymb}
\usepackage{xcolor}
\usepackage[ruled]{algorithm2e}
\usepackage[breaklinks, colorlinks, citecolor=blue, urlcolor=blue, linkcolor=black]{hyperref}
\usepackage{amsmath}
\usepackage[mathscr]{euscript}
\usepackage{longtable}
\usepackage{tabularx}
\usepackage{paralist}
\usepackage{soul}
\usepackage{booktabs}
\usepackage{cite}
\usepackage{subcaption}
\newcommand{\ourparagraph}[1]{%
  \medskip

  \noindent\emph{#1}.
}

\usetikzlibrary{fit}  
\usetikzlibrary{shapes.geometric}

\usepackage{soul} 

\newcommand{\ag}[1]{\todo[color=green!20!yellow, size=\scriptsize]{\textbf{AG:} #1}}

\newcommand{\cev}[1]{\reflectbox{\ensuremath{\vec{\,\,\reflectbox{\ensuremath{#1}}}}}\!\!}
\newcommand{\lbl}{\ell} 
\newcommand{\flbl}{\vec{\ell}} 
\newcommand{\blbl}{\cev{\ell}} 

\begin{document}
\title{Branch and Price for the Length-Constrained Cycle Partition Problem}
%
%
\author{Mohammed Ghannam\inst{1,2,4}\orcidID{0000-0001-9422-7916} \and
Gioni Mexi\inst{2}\orcidID{0000-0003-0964-9802} \and
Edward Lam \inst{3}\orcidID{0000-0002-4485-5014} \and
Ambros Gleixner\inst{1,2}\orcidID{0000-0003-0391-5903}}
\authorrunning{M. Ghannam et al.}
\institute{HTW Berlin, Germany \and
Zuse Institute Berlin, Germany \and
Monash University, Australia \and
felmo GmbH, Germany\\}
\maketitle
\begin{abstract}

    The length-constrained cycle partition problem (LCCP) is
    a graph optimization problem in which a set of nodes must be partitioned into a minimum number of cycles.
    Every node is associated with a critical time and the length of every cycle must not exceed the critical time of any node in the cycle.
    We formulate LCCP as a set partitioning model and solve it using an exact branch-and-price approach.
    We use a dynamic programming-based pricing algorithm to generate improving cycles,
    exploiting the particular structure of the pricing problem for efficient bidirectional search and symmetry breaking.
    Computational results show that the LP relaxation of the set partitioning model produces strong dual bounds and our branch-and-price method improves significantly over the state of the art.
    It is able to solve closed instances in a fraction of the previously needed time and closes $13$~previously unsolved instances,
  one of which has~$76$ nodes, a notable improvement over the previous limit of~$52$ nodes.
\end{abstract}
%
%

\section{Introduction}\label{sec:intro}
\newcommand{\cycles}{\Omega}
\newcommand{\cycle}{C}
\newcommand{\fa}{\text{for all }}
A \emph{cycle partition} of an undirected graph~$G=(V,E)$ is a partition of the set of nodes~$V$ into disjoint cycles. In the \emph{length-constrained cycle partition problem} (LCCP) introduced in~\cite{hoppmann2020minimum},
we are additionally given a \emph{critical time}~$q_i > 0$ for every node~$i \in V$ and a \emph{travel time}~$t_{i,j} \geq 0$ for every edge~$\{i,j\} \in E$.
We call a cycle $\cycle=(i_0,i_1,\ldots,i_K=i_0)$ \emph{length-feasible} if it satisfies the
length constraint
\begin{equation}\label{eq:lc}
t(\cycle) := t_{i_0,i_1} + t_{i_1,i_2} + \ldots + t_{i_{K-1},i_K}
\leq q(\cycle) := \min \{ q_{i_0},q_{i_1},\ldots,q_{i_K} \},
\end{equation}
i.e., the total time to traverse the edges of the cycle must not exceed the critical time
of any node in the cycle.
In other words an agent that continuously travels along the cycle will visit each node~$i$ of the cycle with frequency at least~$q_i$.
The LCCP then amounts to computing the smallest number of disjoint, length-feasible cycles
that cover all nodes.

The setting of LCCP is found to be useful, e.g., in applications where the objective is to find the smallest number of agents to conduct periodic surveillance or maintenance actions~\cite{hoppmann2022mathematical}.
Similar types of length constraints can be found in the literature on minimizing cycle length in kidney exchange programmes~\cite{Lam:2019ab}, where the length is simply the number of nodes in the cycle.

In \cite{hoppmann2020minimum}, Hoppmann-Baum et al. show that the LCCP is NP-hard by reduction from the \emph{traveling salesman problem} (TSP), and that no polynomial-time approximation algorithm exists.
They present a compact mixed-integer programming (MIP) formulation based on the Miller, Tucker, and Zemlin (MTZ)~\cite{miller1960integer} formulation of the TSP. This compact model was able to solve instances with up to 29 nodes to proven optimality.
In \cite{lccp}, the same authors introduced a MIP formulation based on subtour elimination constraints (SEC), valid inequalities derived from cliques in conflict hypergraphs, and an efficient primal heuristic.
A branch-and-cut implementation of the SEC model based on a state-of-the-art commercial
MIP solver was able to solve instances with up to 52 nodes.
%
Despite the mentioned improvements, the proposed models suffer from
symmetry, leading to a weak linear relaxation and ineffective branching.

The goal of our work is to try and overcome several bottlenecks of the previous approaches by considering a set
partitioning formulation based on cycle variables.
Let~$\Omega$ denote the set of all length-feasible cycles, and let~$a_i^\cycle \in \{0,1\} $ be a
constant equal to $1$ iff node~$i$ is contained in cycle~$\cycle$.
Introducing binary variables $\lambda_\cycle$ to indicate whether cycle $\cycle \in \Omega$ is used in the solution, we arrive at the set partitioning formulation
\begin{equation}\label{prob:imp}
\begin{array}{llll}
\min\;\; & \sum_{\cycle \in \Omega} \lambda_\cycle & \\[1ex]
\text{s.t.} & \sum_{\cycle \in \Omega} a_i^\cycle  \lambda_\cycle &= 1  & \quad \fa i \in V, \\[1ex]
              & \mbox{}\hfill\lambda_{\cycle} &\in \{0,1\} & \quad \fa \cycle \in \Omega.
       \end{array}
\end{equation}
This reformulation of LCCP is inspired by formulations for the
cardinality-constrained multi-cycle problem in kidney exchange~\cite{Lam:2019ab}, where a maximum-value packing of cycles is required, and by exact methods in many logistics applications such as vehicle routing, see, e.g.,~\cite{costa_exact_2019}.
As in these applications, it is not viable to solve \eqref{prob:imp} explicitly due to the
exponential number of variables.
Instead, we apply \emph{column generation} in order to solve the linear programming (LP)
relaxation of~\eqref{prob:imp} dynamically, and integrate this into an
exact \emph{branch-and-price} algorithm enabling us to solve LCCP instances to proven
optimality, see, e.g.,~\cite{cg-book,Wolsey2020} for a general overview of this methodology.

A major computational bottleneck in this approach is the
so-called pricing problem which needs to be solved repeatedly in order to dynamically add
variables with negative reduced cost to the LP relaxation of~\eqref{prob:imp}.
Given travel times $t_{i,j} \geq 0$ for $\{i,j\}\in E$ and both critical times $q_i > 0$ and
weights $\pi_i\in\mathbb{R}$ for~$i\in V$, the pricing problem amounts to computing
a cycle $C$ that satisfies the length constraint~\eqref{eq:lc} and
maximizes the sum of node weights.
We call this the \emph{length-constrained prize-collecting cycle
problem} (LCPCCP).
Note that problem data may be non-metric, i.e., travel times $t_{i,j}$ may violate the triangle inequality, and that the node weights~$\pi_i$, which are
derived as dual multipliers associated with the partitioning constraint of the node in the
LP relaxation, are not restricted in sign.
Even in the metric case and when critical times are assumed to be constant, one can show that LCPCCP is strongly NP-hard by reduction from the Hamiltonian cycle problem.

The main contributions of this paper are as follows.
In \autoref{sec:labeling} we present a new dynamic programming algorithm, also known as
label setting, for LCPCCP including dominance rules for reducing the search space and an efficient bidirectional search that allows us to enumerate cycles
only half-way.
In \autoref{sec:bnp} we integrate the resulting column generation scheme into an exact
branch-and-price algorithm enhanced by symmetry breaking, heuristic pricing, parallel
pricing, early branching, and techniques to exploit the triangle inequality for metric
input data.
In \autoref{sec:results} we describe the results of our computational study to analyze the
performance of the new approach on benchmark instances from the literature and to quantify
the impact of the individual improvement techniques.
The full branch-and-price algorithm is able to solve $13$~open instances with up to $76$~nodes, and is on average $14.6$~times faster than the best previous approach.


\section{Dynamic Programming for the Length-Constrained Prize-Collecting Cycle Problem}\label{sec:labeling}
\newcommand{\redcost}{\overline{c}}
\newcommand{\lblnodes}{\mathscr{N}}
\newcommand{\lbldef}{(\lblnodes, v, \redcost, t, q)}

Column generation is an advanced technique for solving large linear programs and is an essential subroutine in branch-and-price algorithms~\cite{cg-book,Wolsey2020}.
The method has recently been found to date back to works by Kantorovich and Zalgaller in $1951$~\cite{uchoa_kantorovich_nodate,KantorovichZalgaller1951}, and was independently developed by Gomory and Gilmore~\cite{GilmoreGomory1961}.

In order to solve the LP relaxation of \eqref{prob:imp} by column
generation, we first note that the upper bounds $\lambda_\cycle\leq 1$ are implied by the
partitioning constraints and can be removed.
We call the resulting LP the \emph{master problem} (MP).
Second, we replace $\Omega$ with a subset of columns $\Omega' \subset \Omega$ to form
the \emph{restricted master problem} (RMP).
Column generation then repeatedly optimizes the RMP, each time producing a primal-dual
pair of solutions~$(\lambda, \pi)$ and solving the so-called \emph{pricing problem} to
check whether cycles~$\cycle$ with negative reduced cost
\begin{equation}
  \label{prob:pp}
  \textstyle\redcost(\cycle) = 1 - \sum_{i \in V} a^\cycle_i \pi_i = 1 - \sum_{i \in C} \pi_i
\end{equation}
exist in $\Omega \setminus \Omega'$.
If yes, these are added to $\Omega'$ and the RMP is
reoptimized; if not, then $(\lambda, \pi)$ is optimal for MP.
We refer to~\cite{cg-book,Wolsey2020} for details on column generation.

As explained above, the pricing problem of minimizing~\eqref{prob:pp} amounts to the length-constrained prize-collecting cycle problem.
%
%
In the following, suppose we restrict LCPCCP by fixing one node $s\in V$ to be
contained in the cycle.
%
%
Then the resulting problem closely resembles a
\emph{resource-constrained shortest path problem} (RCSPP)~\cite{irnich_resource_2008}
typically solved in vehicle routing 
applications on a directed graph to 
find a path of minimum reduced cost
that starts at and returns to a given node
such that resources accumulated along arcs lie within a constant resource interval specified for each node.
%
Similarly, we propose to solve LCPCCP for a fixed start node using dynamic programming as
an implicit enumeration scheme over all length-feasible cycles starting from~$s$.
%

To this end, we form \emph{partial cycles} (paths) and iteratively extend these paths with incident unvisited nodes until $s$ is visited again, essentially closing the cycle.
If at any point a partial cycle becomes infeasible (cannot be extended to a length-feasible cycle) or is proved to be dominated (would only lead to cycles with same or larger reduced cost), it is discarded, see~\autoref{sec:pruning}.
This process is repeated until all non-dominated length-feasible cycles are explored. If there are cycles with negative reduced cost, the ones with minimum reduced cost are returned.



Let~$P=(i_0,i_1, \dotsc, i_K)$ be a path in $G=(V,E)$ with nodes $i_0, \dotsc, i_K \in V$ and edges $\{i_0,i_1\}, \dotsc, \{i_{K-1},i_K\} \in E$.
The
path $P$ can be represented
by a so-called \emph{label}~$\lbl = \lbldef$ where
\begin{itemize}
    \item $\lblnodes = \lblnodes(\lbl) = \{i_1, \dotsc, i_K\}$ is the (unordered) set of nodes without start node,
    \item $v = v(\lbl) = i_K$ is the last node visited in~$P$,
    \item $\redcost = \redcost(\lbl) = 1 - \sum_{i \in \lblnodes } \pi_i$ is the reduced cost of~$P$,
    \item $t = t(\lbl) = \sum_{k=1}^K t_{i_{k-1},i_k} $ is the total travel time taken by~$P$, and
    \item $q = q(\lbl) = \min_{k=0,1, \dotsc, K} q_{i_k}$ is the minimum critical time of all nodes in~$P$.
\end{itemize}

The initial label for a starting node~$s$ is given as~$\lbl_s=( \{\}, s, 1, 0, q_s)$.
During the search, a label~$\lbl = \lbldef$ can be \emph{extended} using any
edge~$\{v(\lbl),j\}$ from the last node to a new node $j \in V \setminus \lblnodes$.
This creates a new label
$\lbl^+$ where $\lblnodes(\lbl^+) = \lblnodes(\lbl) \cup \{j\}$, $v(\lbl^+) = j$, $\redcost(\lbl^+) = \redcost(\lbl) - \pi_j$, $t(\lbl^+) = t(\lbl) + t_{i,j}$, and $q(\lbl^+) = \min\{q(\lbl), q_j\}$.
%
We call a label~$\lbl^* \neq \lbl_s$  \emph{fully-extended} if it represents a
cycle, i.e., if $v(\lbl^*) = s$.
Note that LCCP allows singleton cycles (loops) as part of the solution; these are included
here as label~$(\{s\}, s, 1-\pi_s, 0, q_s)$.
In order to recover the cycle corresponding to a fully-extended label, we also keep track
of each label's predecessor, i.e., the label from which it was extended.

\subsection{Pruning the Search Space by Feasibility and Dominance}
\label{sec:pruning}

Besides the data structures used to represent and extend labels, the efficiency of the
resulting dynamic programming method largely depends on how much of the search space needs
to be explored in order to terminate with a provably optimal cycle.
%
Clearly, a label~$\lbl$ can be discarded if it can be proven to only lead to length-infeasible
cycles. This is the case if~$ t(\lbl) > q(\lbl)$.
Then for any extension~$\lbl^+$ of label~$\lbl$, we have $  t(\lbl^+) \geq t(\lbl) > q(\lbl) \geq q(\lbl^+)$.
By induction, all fully-extended labels~$\lbl^*$ obtained by successive extensions of label~$\lbl$ satisfy~$t(\lbl^*) > q(\lbl^*)$, hence the corresponding cycles violate the length constraint~\eqref{eq:lc}.

An effective technique for pruning the search space further is to exploit dominance relations.
A label~$\lbl_a$ is said to \emph{dominate} label~$\lbl_b$ if they have the same end node~$v=v(\ell_a)=v(\ell_b)$, and all cycles reachable from successive extensions of~$\lbl_a$ have an equal or lower reduced cost than those reachable from successive extensions of~$\lbl_b$.
Then again, by induction, we may discard label~$\lbl_b$ in the search, because at least
one minimum reduced cost cycle must remain.

\begin{lemma}\label{lemma:dominate}
    A label~$\lbl_a$ dominates a label~$\lbl_b$ if~$v(\lbl_a) = v(\lbl_b)$ and the conditions
    \begin{subequations}
        \begin{align}
            \redcost(\lbl_a) &\leq \redcost(\lbl_b), \label{dom-redcost} \\
            t(\lbl_a) &\leq t(\lbl_b), \text{ and} \label{dom-time} \\
            \lblnodes(\lbl_a) &\subseteq \lblnodes(\lbl_b) \label{dom-subset}
        \end{align}
    \end{subequations}
    hold.
\end{lemma}

\begin{proof}
    From~\eqref{dom-subset} we have that the set of possible extensions of label~$\lbl_b$ is a subset of the set of possible extensions of label~$\lbl_a$.
    From~\eqref{dom-subset} it also follows that $q(\lbl_a) = \min_{i\in\lblnodes(\lbl_a)} q_i \geq \min_{i\in\lblnodes(\lbl_b)} q_i = q(\lbl_b)$.
    Together with~\eqref{dom-time} this proves that any feasible extension of~$\lbl_b$ is also feasible for~$\lbl_a$. Therefore, any cycle reachable from~$\lbl_b$ is also reachable from~$\lbl_a$, and from \eqref{dom-redcost} it has an equal or lower reduced cost. Therefore,~$\lbl_a$ dominates~$\lbl_b$.
\end{proof}

\subsection{Bidirectional Search}

Exploring the search space from both forward and backward directions simultaneously is
another successful technique to reduce the number of explored labels.
For the RCSPP, \cite{tilk_asymmetry_2017,tilk_bidirectional_2020} has shown that
bidirectional search can be more efficient than monodirectional search.
The efficiency comes from the fact that the
number of non-dominated labels can grow exponentially with the length of the paths.
Hence, if it is possible to explore forward and backward paths only until a halfway point and merge them to full paths, this can drastically reduce the total number of labels generated.

For LCPCCP, we can apply bidirectional search even more effectively than for the general RCSPP. Since the graph is undirected and the resource consumption by travel times is symmetric, only extensions from
one direction are needed. Specifically, in bidirectional search for a fixed start node~$s$, we only extend labels~$\lbl$
with $t(\lbl) < \frac{q(\lbl)}{2}$.
We then \emph{merge} non-dominated labels as follows.

\newcommand{\concat}{||}
Let labels~$\lbl_a$ and $\lbl_b$ both start at $s$, end at $v=v(\lbl_a) = v(\lbl_b)$, and be otherwise disjoint, i.e., $\lblnodes(\lbl_a) \cap \lblnodes(\lbl_b) = \{v\}$.
Then we create a cycle by concatenating the path represented by~$\lbl_a$ and, in reverse order, the path represented by~$\lbl_b$.
%
We denote the newly merged label by $\lbl_a\concat\lbl_b$, given by
$\lblnodes(\lbl_a\concat\lbl_b) = \lblnodes(\lbl_a) \cup \lblnodes(\lbl_b) \cup \{s\}$, $v(\lbl_a\concat\lbl_b) = s$, $\redcost(\lbl_a\concat\lbl_b) = \redcost(\lbl_a) + \redcost(\lbl_b) - \pi_s + \pi_{v(\lbl_a)} - 1$,
    $t(\lbl_a\concat\lbl_b) = t(\lbl_a) + t(\lbl_b)$, and
    $q(\lbl_a\concat\lbl_b) = \min\{q(\lbl_a), q(\lbl_b)\}$, see~\autoref{fig:merge} for an example.

\begin{figure}[ht]
    \centering
    \begin{tikzpicture}
        \node[draw, circle, label=180:$\lbl_a$] (11) at (0,0) {1};
        \node[draw, circle] (3) at (1,0) {3};
        \node[draw, circle] (5) at (2,0) {5};

        \node[draw, circle, label=0:$\lbl_b$] (12) at (6,0) {1};
        \node[draw, circle] (4) at (5,0) {4};
        \node[draw, circle] (2) at (4,0) {5};

        \draw[->, black] (11) -- (3);
        \draw[->, black] (3) -- (5);

        \draw[->, black] (12) -- (4);
        \draw[->, black] (4) -- (2);

        \node[draw, circle, label=180:$\lbl_a \concat \lbl_b$] (m11) at (0,-1) {1};
        \node[draw, circle] (m3) at (1,-1) {3};
        \node[draw, circle] (m5) at (3,-1) {5};
        \node[draw, circle] (m4) at (5,-1) {4};
        \node[draw, circle] (m12) at (6,-1) {1};

        \draw[->, black] (m11) -- (m3);

        \draw[->, black, thick] (m3) -- (m5);
        \draw[->, black, thick] (m5) -- (m4);
        \draw[->, black, thick] (m4) -- (m12);
    \end{tikzpicture}
\caption{Example of merging two partial cycles.}
\label{fig:merge}
\end{figure}
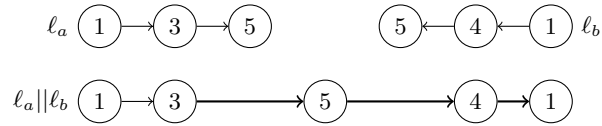


\begin{lemma}\label{lemma:merge}
   Dynamic programming with dominance and bidirectional search finds two labels~$\lbl_a$ and~$\lbl_b$ such that $\lbl_a\concat\lbl_b$ represents a length-feasible cycle with minimum reduced cost.
\end{lemma}
\begin{proof}
    Let $C=(i_0,i_1,\ldots,i_K)$ be a cycle with minimum reduced cost that is found by some version of full, monodirectional forward search with dominance checks.
    Suppose $\mathscr{\vec{L}}_C = (\flbl_0, \flbl_1, \ldots, \flbl_{K} )$ is the sequence of labels leading to~$C$ in this search, then choose~$\flbl_{k}$ to be the first label in~$\mathscr{\vec{L}}_C$ where~$t(\flbl_{k}) \geq q(\flbl_{k})/2$ is satisfied.
    Hence, label~$\lbl_a:=\flbl_{k}$ is also explored in bidirectional search.
    It represents the path $(i_0,i_1,\ldots,i_k)$.

Next, consider a sequence of labels that would generate the same cycle, but in reverse order, i.e., $(i_K,i_{K-1},\ldots,i_0)$.
Denote this sequence of labels by $\cev{\mathscr{L}}_C = ( \blbl_K, \blbl_{K-1}, \ldots, \blbl_{0} )$,
    then $\blbl_{k}$ represents the path $(i_K,i_{K-1},\ldots,i_k)$.
    From $t(C) = t(\flbl_{k})+t(\blbl_{k}) \leq q(C)$ and $t(\flbl_{k}) \geq q(\flbl_{k})/2$ it follows that
    \[
    t(\blbl_{k}) = t(C) - t(\flbl_{k}) \leq q(C) - \frac{q(\flbl_{k})}{2} \leq q(C) - \frac{q(C)}{2} = \frac{q(C)}{2} \leq \frac{q(\blbl_{k})}{2},
    \]
hence $\blbl_{k}$ is length-feasible and within the halfway point.
Therefore, unless it is dominated, $\blbl_{k}$ is found by bidirectional search, and we can choose $\lbl_{b} := \blbl_{k}$, since $\lbl_{a} \concat \lbl_{b}$ represents~$C$.
If $\blbl_{k}$ is discarded due to dominance, then (by transitivity of dominance) bidirectional search must find another label that dominates $\blbl_{k}$.
By the conditions of Lem.~\ref{lemma:dominate} we show that $\lbl_b$ can be chosen as this dominating label.

Labels $\lbl_a$ and $\lbl_b$ trivially start at the same node~$s=i_0$, 
and due to $v(\lbl_b) = v(\blbl_k)=i_k$, they also end at the same node $i_k$.
Due to $\lblnodes(\lbl_b) \subseteq \lblnodes(\blbl_{k})$, we have
$\lblnodes(\lbl_a) \cap \lblnodes(\lbl_b) \subseteq \lblnodes(\lbl_{a}) \cap \lblnodes(\blbl_{k}) = \{i_k\}$, i.e., they do not overlap otherwise.
Hence, they can be merged to $\lbl_a\concat\lbl_b$.
It remains to show that this merged label is length-feasible and has the same reduced cost as $C$.
This follows from the fact that $\lbl_a\concat\blbl_k$ is feasible and optimal, and $t(\lbl_b) \leq t(\blbl_{k})$, $q(\lbl_b) \geq q(\blbl_{k})$, and $\redcost(\lbl_b) \leq \redcost(\blbl_{k})$ hold by domination.

\end{proof}


Note that bidirectional search may produce more cycles than monodirectional search. The reason is that many of the labels at the halfway point,
could be dominated by labels found later in monodirectional search.
Adding these potentially sub-optimal cycles to the RMP can aid the convergence of column generation but it can make re-optimizing much harder. We try to safe-guard against this effect by sorting the cycles by reduced cost and returning at most a fixed number ($50$ in our implementation) of the ones with lowest reduced cost.



\section{Branch-and-Price for the Length-Constrained Cycle Partition Problem}\label{sec:bnp}
In the following, we discuss different aspects of how we extend the column generation method entailed by \autoref{sec:labeling} into an exact branch-and-price algorithm.
%
%

\ourparagraph{Branch and Bound}
We use a standard branching strategy based on implicit edge variables as described in \cite{costa_exact_2019}.
%
%
We branch on the most used edge variable that does not belong to a singleton cycle.
The rationale is that the most used edge would create the biggest disturbance in the subproblem, leading to an early fathoming of the subproblem due to infeasibility or finding a feasible integer solution.
For branch-and-bound node selection, we employ the best-estimate rule with plunging
described in \cite{achterberg_constraint_2007}.
After selecting a branch-and-bound node, we solve the LP relaxation using column generation. In each pricing iteration
the dynamic programming algorithm from \autoref{sec:labeling} is called from each node starting node $s\in V$.
As these calls are independent they are run in parallel.

\ourparagraph{Symmetry Breaking}
For each call to the dynamic program, i.e., from a fixed starting node $s$,
the search can safely ignore all nodes with an index less than the starting node.
Since each cycle contains a unique node with the lowest index, this restriction prevents the generation of an equivalent cycle with a different starting node. A further improvement on this restriction is to sort the node indices by their critical time, such that the pricing problem that explores the most nodes starts from the smallest critical time, allowing earlier pruning due to feasibility.

\ourparagraph{Heuristic and Exact Pricing}
We start each pricing call with a faster, incomplete dynamic programming algorithm that
relaxes the subset conditions in the dominance rule \eqref{dom-subset}.
If this fails to find a cycle with negative reduced cost, then exact pricing is performed.
From exact pricing we always compute a Lagrangian lower bound~\cite{cg-primer}.
Let~$z^*_s$ be the optimal objective value of the pricing problem for a fixed starting node~$s$, and let~$z^*_{MP}$, $z^*_{RMP}$ be the optimal objective value of MP and RMP, respectively, then $z^*_{MP} \geq z^*_{RMP} + \sum_{s \in V} z^*_s =: LB_{lg}$.
From integrality of the objective function, we can use $\lceil LB_{lg} \rceil$ as a valid lower bound.

\ourparagraph{Early Branching}
Another acceleration idea based on integrality, inspired by the technique in~\cite{coloring}, is to skip pricing in some branch-and-bound node $Q$ by looking at the lower bound $\lceil LB_{lg} \rceil$ of the parent node.
Let $z^*_{RMP}$ be the first RMP objective value for $Q$ obtained after removing columns due to branching. If~$\lceil z^*_{RMP} \rceil = \lceil LB_{lg} \rceil$, then computing $z^*_{MP}$ exactly will not improve over the lower bound of the parent node. Therefore, we can skip pricing in this subproblem and perform early branching.

\ourparagraph{Farkas Pricing and RMP Initialization}
If the RMP at a subproblem becomes infeasible due to branching and removal of columns, we use Farkas pricing \cite{Nunkesser2006,CeselliGLNS08} to generate new columns that render the RMP feasible again, or to prove that the MP is infeasible.
The root node RMP is always feasible, because we initialize it with the trivially feasible singleton cycles and the cycles in the primal solution generated by the Most-Critical-Vertex-Based Heuristic (MCV) from~\cite{lccp}.

\ourparagraph{Exploiting the Triangle Inequality}
The algorithm developed up to this point works for any choice of nonnegative travel times.
For many real-world applications, such as in aerial vehicle routing where distances are Euclidean, we know additionally that the triangle inequality holds, i.e., that
$t_{i,j} \leq t_{i,k} + t_{k,j}$ for all~$\{i,j\},\{i,k\},\{j,k\} \in E$.
Under this assumption, the set partitioning formulation can be turned into a set covering formulation, which is easier to solve because the dual solution $\pi$ becomes nonnegative.
A solution for the set partitioning formulation can always be retrieved from a set covering solution by removing nodes that appear in more than one cycle from all but one cycle. This is guaranteed to be feasible since removing a node could only decrease the total travel time and increase the minimum critical time of the cycle.
By the same arguments, we can prove the following lemma that show that nodes~$i\in V$ with $\pi_i=0$ can be temporarily ignored during dynamic programming. 
\begin{lemma}\label{lemma:tri-eq}
  Suppose travel times satisfy the triangle inequality, and we are given a dual solution $\pi\in\mathbb{R}^V$, then there exists an optimal cycle~$C$ to the pricing problem $\max_{C\in\Omega}\sum_{i \in C} \pi_i$ such that~$\pi_i > 0$ for all~$i \in C$.
\end{lemma}


\label{subsec:tri-ineq}


\section{Computational Results}\label{sec:results}
We implemented the branch-and-price algorithm described in \autoref{sec:bnp} using the
Python wrapper of the branch-and-price framework
SCIP~\cite{pyscipopt,BestuzhevaBesanconEtal2023_Enabling}, interfacing to a
Rust implementation for solving the pricing problem.
It calls the dynamic programming method from \autoref{sec:labeling} from each starting node in parallel using the Rayon Rust library~\cite{rayon_rs}. 
The rest of SCIP is sequential.
%
%
All experiments were run with a time limit of two hours on a cluster of identical machines equipped
with Intel(R) Xeon(R) Gold 5122 processors with 3.6\,GHz and 96\,GB of RAM; each chip had 8~cores, and allowed for 16~threads to run in parallel.
The code and the instance-wise results of the experiments are publicly
available on GitHub at \url{https://github.com/mmghannam/lccp-bnp}.

The goal of our experiments was to analyze the performance of our branch-and-price
algorithm, in particular to answer the following questions:
\begin{enumerate}
\item How does the new method compare to the previous approach from~\cite{lccp} in terms of speed and ability to solve instances to optimality?
\item How tight is the linear relaxation of the set partitioning formulation in general and in comparison to the compact formulation from~\cite{lccp}?
\item What is the performance impact of the individual acceleration techniques?
\end{enumerate}%
Moreover,
we
were curious whether it matters if the travel times
respect the triangle inequality, and how the performance of the
algorithms is affected if we modify travel times in the above
instances to satisfy the triangle inequality.

We use the standard benchmark set from~\cite{lccp},
which consists of 84~instances that have between 14 and 100~nodes.
All algorithms evaluated were initialized with a solution provided by the MCV heuristic from~\cite{lccp}. 
We compare results for three major algorithms:
\begin{itemize}
    \item \textsc{bnc-sec}: branch and cut with subtour elimination constraints from~\cite{lccp} based on Gurobi~10.0.3~\cite{gurobi}
    \item \textsc{bnp-basic}: plain branch and price without any improvement technique
    \item \textsc{bnp-full}: branch and price with all improvement techniques
\end{itemize}
%
Additionally, we ran the following variations of \textsc{bnp-full} to quantify the impact of four components:
\begin{itemize}
    \item \textsc{bnp-nobidir}: branch and price with monodirectional labeling
    \item \textsc{bnp-nopar}: branch and price without parallelization
    \item \textsc{bnp-nosymbr}: branch and price without symmetry breaking
    \item \textsc{bnp-noearly}: branch and price without early branching
\end{itemize}
Note that the labeling algorithm in \textsc{bnp-basic} already performs dominance checks,
since otherwise not more than the smallest instances could be solved.

In Tab.~\ref{tab:aggregated} we report the shifted geometric mean of the solving times for all instances solved by at least one algorithm.
We use a shift of 1~second;
time outs are included with the time limit.
Column~``Ratio'' states the relative solving time w.r.t. the fastest
solver \textsc{bnp-full}.
In Fig.~\ref{fig:plot} we plot the number of solved instances over time for each
algorithm.

\ourparagraph{Overall Results}
We can first observe that \textsc{bnp-full} drastically outperforms \textsc{bnc-sec}, on
average being more than an order of magnitude ($14.6$ times) faster.
\textsc{bnp-full} is able to solve $13$~more instances to optimality.
The largest instance solved features $76$~nodes compared to $52$~nodes previously.
Notably, all $39$~instances solved by \textsc{bnc-sec} are also solved
by \textsc{bnp-full}.

\ourparagraph{LP Relaxations}
%
While \textsc{bnc-sec} could solve the root relaxation (including subtour elimination) on all
$84$~instances, the column generation loop of \textsc{bnp-full} converged to an optimal
root LP solution for only $56$~instances, running into the memory or time limit in the
remaining cases.
Nonetheless, whenever computed successfully, the root dual bound of the set partitioning formulation proves to be significantly tighter than the root dual bound of \textsc{bnc-sec}.
For the instances where both algorithms could solve the linear relaxation at the root
node, the average gap with respect to the best known primal solution is $1.1\%$ for \textsc{bnp-full} vs.~$23.6\%$
for \textsc{bnc-sec}.
In addition, on instances for which the optimal objective value $z^*$ is known the root
dual bound of the set partitioning formulation (rounded up due to integrality) is at least
$z^* - 1$.
This finding suggests that the set-partitioning model may have a round-up property similar to that investigated in~\cite{scheithauer1995modified}.

\ourparagraph{Primal Bounds}
In \cite{lccp} it is reported that the MCV heuristic is effective in producing good, and often optimal, solutions. Specifically, it finds the optimal solution for $21$ out of the $39$ instances that are solved by both \textsc{bnc-sec} and \textsc{bnp-full}. For the $13$ instances newly solved by \textsc{bnp-full}, the heuristic finds the optimal solution for only $2$ of them, displaying the importance of the exact branch-and-price run also on the primal side. Interestingly, for instances where both algorithms time out, the primal bounds imposed by the MCV heuristic are hard to improve on. Only \textsc{bnp-full} manages to improve the primal bound for $2$ of these instances, by a value of $1$.

\ourparagraph{Impact of the Improvement Techniques}
Bidirectional labeling is the most effective amongst the improvement techniques.
In comparison, monodirectional labeling generates substantially more labels and results in a $4.4$ times slowdown overall.
%
The remaining improvement techniques also contribute significantly to the performance, although to a lesser degree.
While disabling symmetry breaking and early branching each lead to $7$ fewer instances
solved, parallelization only helps to solve two more instances to optimality.
Last but not least, the importance of the improvement techniques is underlined by the fact
that \textsc{bnp-basic}, the algorithm without any improvement techniques, solves
significantly fewer instances and is slower than \textsc{bnc-sec}.
\begin{figure}[t]
    \centering
    \begin{minipage}[b]{0.39\linewidth}
          \small
  \begin{tabular}{lrrr}
    \toprule
    Solver & \hspace*{-1.5em} Time [s] & Ratio & Solved \\
    \midrule
    \textsc{bnc-sec} & 142.8 & 14.6 & 39 \\
    \textsc{bnp-basic} & 189.1 & 19.3 & 30 \\
    \textsc{bnp-full} & \textbf{9.8} & \textbf{1.0} & \textbf{52} \\
    \midrule
    \textsc{bnp-nobidir} & 43.1 & 4.4 & 40 \\
    \textsc{bnp-nopar} & 12.5 & 1.3 & 50 \\
    \textsc{bnp-nosymbr} & 19.6 & 2.0 & 45 \\
    \textsc{bnp-noearly} & 16.0 & 1.6 & 45 \\
    \bottomrule
  \end{tabular}

        \captionof{table}{Aggregated results.}
        \label{tab:aggregated}
    \end{minipage}
    \hfill
    \hspace{0.5cm}
    \begin{minipage}[b]{0.55\linewidth}
        \includegraphics[width=\linewidth]{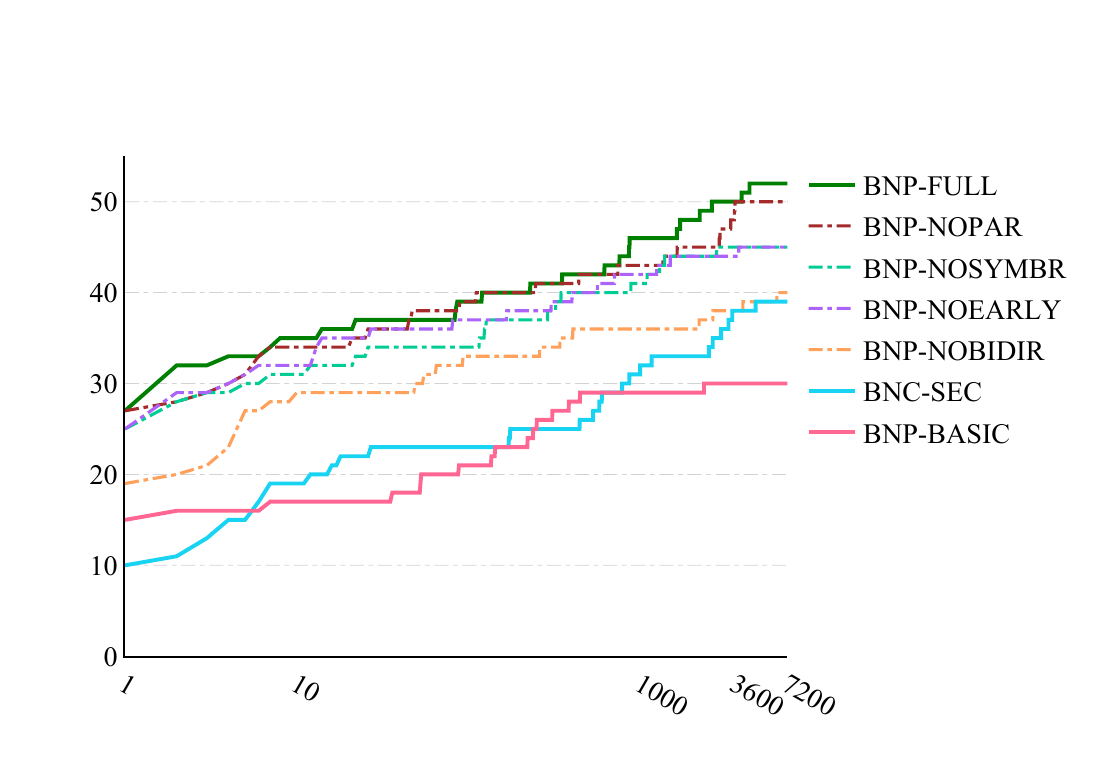}
        \vspace{-.9cm}
        \captionof{figure}{Instances solved over time.}
        \label{fig:plot}
    \end{minipage}
    \vspace{-2em}
\end{figure}
\ag{For plot.pdf: reduce width so it is not so crammed to the table (simply tighter scale on x axis); space between ``Time'' and ``[s]''; start y axis at 0; capital letters in legend and larger font size; optional: axis also at top and right (or arrows on the axis); y axis as long as the table height; thin grey horizontal lines at 10,20,30,40,50 instances; same font as text; order legend like solved instances at 7200s (unsure)}
\ag{I guess you can also remove the axis labels, that makes it less crammed and should be clear from the caption.}

\ourparagraph{Triangle Inequality}
Finally, we modified the $84$~instances from~\cite{lccp} in order to create a new set of
instances that respect the triangle inequality by replacing the travel time on each by the
distance of the shortest paths between its nodes.
Curiously, this change makes the instances harder to solve for both approaches.
On the new benchmark set imposing the triangle inequality, the mean solving time is
increased to $15.0$~seconds for \textsc{bnp-full} and $187.1$~seconds
for \textsc{bnc-sec}.
\textsc{bnp-full} solves $4$~instances less and \textsc{bnc-sec} solves $1$~instance
more than on the original test set.
These slowdowns may be explained by the shorter edge lengths leading to a larger number of
length-feasible cycles.

Note that the above results for \textsc{bnp-full} do not yet exploit the triangle
inequality.
The improvements described in \autoref{sec:bnp}, i.e., to use a set covering formulation
and to remove nodes with zero duals during labeling, enable it to solve $7$~more instances
($55$ in total) and reduce the average solving time to $7.5$~seconds.
This is less than on the original test set, and around $25.2$~times faster
than \textsc{bnc-sec}.
%


\ourparagraph{Conclusion}
To summarize, our branch-and-price algorithm for the length-constrained cycle partition
problem proves to outperform the state of the art significantly, but the results also show
that this is only possible with the help of some nontrivial improvement techniques.
These allow us to close $13$~open instances from the standard benchmarks, and to scale to
larger instances with up to $76$~nodes, compared to $52$~nodes for the previous
branch-and-cut method.
Future work should investigate whether the use of cutting planes can improve the
performance of the branch-and-price algorithm even further, and whether the improvement
techniques developed for LCCP are also valuable for similar problems such as kidney
exchange.

\bibliographystyle{splncs04}
\bibliography{bibliography}
\end{document}